\documentclass[11pt]{amsart}

\usepackage{latexsym}
\usepackage{amsmath}
\usepackage{amssymb}
\usepackage{tikz}
\usepackage{subcaption}
\usepackage{todonotes}
\usetikzlibrary{positioning, shapes.geometric}

\newtheorem{theorem}{Theorem}[section]
\newtheorem{lemma}[theorem]{Lemma}

\newtheorem{proposition}[theorem]{Proposition}
\newtheorem{sublemma}{}[theorem]

\newcommand{\ba}{\backslash}

\newcommand{\cC}{{\mathcal C}}
\newcommand{\cM}{{\mathcal M}}
\newcommand{\cN}{{\mathcal N}}

\newcommand{\cS}{{\mathcal S}}
\newcommand{\cT}{{\mathcal T}}

\usepackage{enumerate}

\usepackage{xfrac}

\begin{document}

\title[Maximum Number of Trees Displayed by a Tree-Child Network]{Note on the Maximum Number of Trees Displayed by a Tree-Child Network}

\author{Yukihiro Murakami}
\address{Delft Institute of Applied Mathematics, Delft University of Technology, Mekelweg 4, 2628 CD, Delft, The Netherlands}
\email{y.murakami@tudelft.nl}

\author{Charles Semple}
\address{School of Mathematics and Statistics, University of Canterbury, Christchurch, New Zealand}
\email{charles.semple@canterbury.ac.nz}

\thanks{The second author was supported by the New Zealand Marsden Fund.}


\date{\today}

\begin{abstract}
In this note, we show that, for all $n\ge 2$, the number of distinct rooted binary phylogenetic $X$-trees displayed by a binary tree-child network $\cN$ on $X$ with $n$ leaves is at most $2^{n-1}-1$ and that this upper bound is sharp. Furthermore, if $\cN$ displays exactly $2^{n-1}-1$ such trees, then exactly one rooted binary phylogenetic $X$-tree is displayed twice, and this tree can be canonically found by iteratively replacing a reticulated cherry with a cherry.
\end{abstract}

\maketitle

\section{Introduction}

Phylogenetic trees are used to represent the evolutionary history of a collection of present-day species. However, because of reticulate (non-treelike) processes such as lateral gene transfer and hybridisation, phylogenetic networks are often a more faithful representation~\cite{doo99}. Amongst the wild assortment of introduced classes of phylogenetic networks, normal networks~\cite{wil08, wil10} and tree-child networks~\cite{car09} are arguably the most prominent and possibly the most important~\cite{fra25}. (Normal networks are tree-child networks with no ``short cuts''.) Both of these classes are general enough to capture complex reticulate evolution, but sufficiently constrained to be meaningful.

Although at the species level evolution is not necessarily treelike, at the level of genes it is typically assumed that evolution follows a treelike process. Thus a phylogenetic network is often viewed as an amalgamation of gene trees. Taking this viewpoint, a natural problem is to enumerate the size of the set $T(\cN)$ consisting of the distinct rooted phylogenetic trees displayed (embedded) in a given phylogenetic network $\cN$. (Here, and throughout the paper, all phylogenetic networks and rooted phylogenetic trees are binary.) In general, computing $|T(\cN)|$ is \#P-complete~\cite{lin13}. However, if $\cN$ is normal and has $k$ reticulations, then $|T(\cN)|=2^k$, but it remains an open problem whether or not $|T(\cN)|$ can be computed in polynomial time if $\cN$ is tree-child. The difficulty with this problem is that a tree-child network can display the same rooted phylogenetic tree in two or more distinct ways. Nevertheless, it was recently shown that if $\cN$ is a tree-child network with $n\ge 3$ leaves, $0\le k\le n-1$ reticulations, and no (underlying) $3$-cycles, then
$$|T(\cN)|\ge 2^{\sfrac{k}{2}} $$
if $k$ is even, and
$$|T(\cN)|\ge \frac{3}{2\sqrt{2}} 2^{\sfrac{k}{2}}$$
if $k$ is odd, and that these lower bounds are sharp~\cite{sem26}. (Note that a tree-child network on $n$ leaves has at most $n-1$ reticulations.) So what are the analogous upper bounds for tree-child networks? If $\cN$ is a tree-child network with $n\ge 2$ leaves and $0\le k\le n-2$ reticulations, then, as for normal networks, $|T(\cN)|\le 2^k$. But what if $\cN$ has $n-1$ reticulations? The purpose of this note is to establish the following theorem.

\begin{theorem}
Let $\cN$ be a tree-child network with $n$ leaves, where $n\ge 2$. Then
$$|T(\cN)|\le 2^{n-1}-1.$$
Moreover, this upper bound is sharp.
\label{main}
\end{theorem}

The proof of the inequality in Theorem~\ref{main} is not difficult and showing that it is sharp is really just a matter of finding a family of tree-child networks that reach this bound for all $n\ge 2$. However, the result is not what we expected or assumed. Given the prominent role tree-child networks are playing in mathematical phylogenetics, it is for these reasons that we have written this paper.

The paper is organised as follows. The next section contains some preliminaries including the statements of several results on tree-child networks. The proof of Theorem~\ref{main} is established in Section~\ref{proof}. If $\cN$ is a tree-child network with $n$ leaves and $|T(\cN)|=2^{n-1}-1$, then $\cN$ displays exactly one rooted phylogenetic tree $\cT$ in two distinct ways. We complete Section~\ref{proof} by showing that $\cT$ can always be canonically found by iteratively replacing a reticulated cherry with an (ordinary) cherry.

\section{Preliminaries}
\label{preliminaries}

Throughout the paper $X$ denotes a non-empty finite set.

\noindent {\bf Phylogenetic networks.} A {\em phylogenetic network $\cN$ on $X$} is a rooted acyclic directed graph with no parallel arcs satisfying the following properties:
\begin{enumerate}[(i)]
\item the (unique) root has out-degree two;

\item the set of vertices of out-degree zero is $X$ and each vertex in $X$ has in-degree one; and

\item all other vertices either have in-degree one and out-degree two or in-degree two and out-degree one.
\end{enumerate}
For technical reasons, if $|X|=1$, then $\cN$ consists of a single vertex labelled by the element in $X$. The elements in $X$ are called {\em leaves}. The vertices of in-degree one and out-degree two are {\em tree vertices}, while the vertices of in-degree-two and out-degree one are {\em reticulations}. The arcs directed into a reticulation are called {\em reticulation arcs}. All other arcs are {\em tree arcs}. A reticulation arc $(u, v)$ is a {\em shortcut} if there is a directed path from $u$ to $v$ that avoids traversing $(u, v)$. A {\em rooted phylogenetic $X$-tree} is a phylogenetic network on $X$ with no reticulations. Strictly speaking, what we have called a phylogenetic network and a rooted phylogenetic tree is typically referred as a ``binary phylogenetic network'' and a ``rooted binary phylogenetic tree'', respectively.

Let $\cN$ be a phylogenetic network on $X$, and let $a$ and $b$ be distinct elements in $X$. Let $p_a$ and $p_b$ denote the parents of $a$ and $b$, respectively. If $p_a=p_b$, then $\{a, b\}$ is a {\em cherry}. In this case, we denote the operation of deleting $b$ and suppressing the resulting degree-two vertex by $\cN\ba b$. Note that if $p_a$ is the root of $\cN$, then $\cN\ba b$ denotes the phylogenetic network obtained by deleting $b$ as well as deleting the root of $\cN$. Furthermore, if $p_b$ is a reticulation and $p_a$ is the parent of $p_b$, then $(b, a)$ is a {\em reticulated cherry} in which $b$ is the {\em reticulation leaf}. In general, if $e$ is a reticulation arc of $\cN$, we denote the operation of deleting $e$ and suppressing the resulting degree-two vertices by $\cN\ba e$. 
Again, if $e$ is incident with the root of $\cN$, then $\cN\ba e$ denotes the directed graph obtained by deleting $e$, suppressing the resulting degree-two vertex, and deleting the root. It is worth noting that after deleting a reticulation arc, the resulting directed graph may not be a phylogenetic network. However, if $\cN$ is tree-child, then $\cN\ba e$ is also tree-child (see Lemma~\ref{deletion}).

Let $\cN$ be a phylogenetic network on $X$ and let $\cT$ be a rooted phylogenetic $X$-tree. We say that $\cN$ {\em displays} $\cT$ if a subdivision of $\cT$ can be obtained from $\cN$ by deleting vertices and arcs. Such a subdivision is an {\em embedding} of $\cT$ in $\cN$, in which case $\cT$ is the {\em realisation} of this embedding. We say that $\cN$ {\em displays $\cT$ twice} if there exists two distinct embeddings in $\cN$ that realise~$\cT$. Without ambiguity, we view an embedding of a phylogenetic tree as a subset of the arc set of $\cN$. We write~$T(\cN)$ to denote the set of all displayed trees of $\cN$.
An embedding of $\cT$ in $\cN$ contains at most one of the two reticulation arcs directed into a reticulation, and so, if $\cN$ has $k$ reticulations, then $|T(\cN)|\le 2^k$.

\noindent {\bf Phylogenetic trees.} Let $\cT$ be a rooted phylogenetic $X$-tree, and let $a$, $b$, and $c$ be distinct elements of $X$. The $3$-tuple $(a, b, c)$, or equivalently $(b, a, c)$ is a {\em rooted triple} of $\cT$ if the (underlying unrooted) path from $a$ to $b$ avoids intersecting the path from $c$ to the root of $\cT$. Furthermore, two rooted phylogenetic $X$-trees $\cT_1$ and $\cT_2$ are {\em isomorphic} if there is a map $\varphi: V(\cT_1)\rightarrow V(\cT_2)$ such that $\varphi(x)=x$ for all $x\in X$, and $(u, v)\in E(\cT_1)$ if and only if $(\varphi(u), \varphi(v))\in E(\cT_2)$.

\noindent {\bf Tree-child networks.} Let $\cN$ be a phylogenetic network on $X$. We say that $\cN$ is {\em tree-child} if every non-leaf vertex of $\cN$ is the parent of a tree vertex or a leaf. A tree-child network is {\em normal} if it has no short cuts. In general, the size of $X$ does not bound the total number of reticulations of a phylogenetic network but, for tree-child networks, we have the following.

\begin{proposition}
Let $\cN$ be a tree-child network with $n$ leaves. Then $\cN$ has at most $n-1$ reticulations.
\label{n-1}
\end{proposition}

The next three lemmas are used in the proof of Theorem~\ref{main}. The first is well known and straightforward to prove (see, for example, \cite{bor16}), the second is established in~\cite{doe21}, and the third follows from~\cite[Theorem 1.1]{sem16}.

\begin{lemma}
Let $\cN$ be a tree-child network. Then $\cN$ has either a cherry or a reticulated cherry.
\label{cherry}
\end{lemma}

\begin{lemma}
Let $\cN$ be a tree-child network on $X$ and let $e$ be a reticulation arc of $\cN$. Then $\cN\ba e$ is a tree-child network on $X$.
\label{deletion}
\end{lemma}

\begin{lemma}
Let $\cN$ be a tree-child network on $X$ and let $E$ be a subset of the arcs of $\cN$. Then $E$ is an embedding of a rooted phylogenetic $X$-tree displayed by $\cN$ if and only if $E$ consists of all tree arcs of $\cN$ and, for each reticulation $v$, exactly one reticulation arc directed into $v$.
\label{embeddings}
\end{lemma}

\section{Maximum Displaying Tree-Child Networks}
\label{proof}

In this section, we prove Theorem~\ref{main}. Furthermore, given a tree-child network $\cN$ on~$n\ge 3$ leaves which displays~$2^{n-1}-1$ rooted phylogenetic trees, we show how to canonically find the rooted phylogenetic tree displayed twice by $\cN$ (see Corollary~\ref{cor:DuplicateTree}). Throughout this section, we freely use Lemma~\ref{embeddings}.

\begin{figure}
	\centering
	\begin{subfigure}[b]{0.34\textwidth}
		\centering
		\begin{tabular}[b]{c}
			\resizebox{\textwidth}{\textwidth}{%
				
			\begin{tikzpicture}[
				scale=1, 
				treevertex/.style={circle, fill=black, inner sep=1.8pt},
				leafvertex/.style={circle, fill=black, inner sep=1.8pt},
				reticvertex/.style={rectangle, fill=black, inner sep=2.5pt},
				treeedge/.style={thick},
				reticedge/.style={thick, dashed}
				]

				\node[treevertex, label=above:$p_4$] (p4) at (2.4,5) {};
				\node[treevertex, label=above right:$p_3$] (p3) at (3,4.5) {};
				\node[treevertex, label=above right:$p_2$] (p2) at (3.6,4) {};
				\node[treevertex, label=above right:$p_1$] (p1) at (4.2,3.5) {};
				\node[treevertex, label=right:$q_4$] (q4) at (4.6,3.0) {};
				\node[leafvertex, label=below:$\ell_5$] (l5) at (4.6,0) {};
				
				\node[reticvertex, label=above left:$v_4$] (v4) at (0.7,2.5) {};
				\node[treevertex, label=above left:$q_3$] (q3) at (0.4,2.0) {};
				\node[treevertex, label=above left:$q_2$] (q2) at (0.1,1.6) {};
				\node[treevertex, label=above left:$q_1$] (q1) at (-0.2,1.2) {};
				\node[leafvertex, label=below:$\ell_4$] (l4) at (-0.2,0) {};
				
				\node[leafvertex, label=below:$\ell_3$] (l3) at (1.0,0) {};
				\node[leafvertex, label=below:$\ell_2$] (l2) at (2.2,0) {};
				\node[leafvertex, label=below:$\ell_1$] (l1) at (3.4,0) {};
				
				\node[reticvertex, label=right:$v_3$] (v3) at (1.0,1.6) {};
				\node[reticvertex, label=right:$v_2$] (v2) at (2.2,0.95) {};
				\node[reticvertex, label=right:$v_1$] (v1) at (3.4,0.4) {};
				
				\draw[treeedge] (p4) -- (p3) -- (p2) -- (p1) -- (q4) -- (l5);
				\draw[treeedge] (p4) -- (v4) -- (q3) -- (q2) -- (q1) -- (l4);
				\draw[treeedge] (p3) -- (v3) -- (l3);
				\draw[treeedge] (p2) -- (v2) -- (l2);
				\draw[treeedge] (p1) -- (v1) -- (l1);
				\draw[treeedge] (q3) -- (v3);
				\draw[treeedge] (q2) -- (v2);
				\draw[treeedge] (q1) -- (v1);
				
				\draw[treeedge] (q4) -- (v4);
				
			\end{tikzpicture}
		}
		\end{tabular}
		\caption{$\cN_5$}
		\label{fig:network}
	\end{subfigure}
	\hfill
	\begin{subfigure}[b]{0.34\textwidth}
		\centering
		\begin{tabular}[b]{c}
			\resizebox{\textwidth}{\textwidth}{%
				
			\begin{tikzpicture}[
				scale=1, 
				treevertex/.style={circle, fill=black, inner sep=1.8pt},
				leafvertex/.style={circle, fill=black, inner sep=1.8pt},
				reticvertex/.style={rectangle, fill=black, inner sep=2.5pt},
				treeedge/.style={thick},
				reticedge/.style={thick, dashed}
				]

				\node[treevertex, label=above:$p_4$] (p4) at (2.4,5) {};
				\node[treevertex, label=right:$q_4$] (q4) at (4.6,3.0) {};
				\node[leafvertex, label=below:$\ell_5$] (l5) at (4.6,0) {};
				
				\node[reticvertex, label=above left:$v_4$] (v4) at (0.7,2.5) {};
				\node[treevertex, label=above left:$q_3$] (q3) at (0.4,2.0) {};
				\node[treevertex, label=above left:$q_2$] (q2) at (0.1,1.6) {};
				\node[treevertex, label=above left:$q_1$] (q1) at (-0.2,1.2) {};
				\node[leafvertex, label=below:$\ell_4$] (l4) at (-0.2,0) {};
				
				\node[leafvertex, label=below:$\ell_1$] (l3) at (1.0,0) {};
				\node[leafvertex, label=below:$\ell_2$] (l2) at (2.2,0) {};
				\node[leafvertex, label=below:$\ell_3$] (l1) at (3.4,0) {};
				
				
				\draw[treeedge] (p4) -- (q4) -- (l5);
				\draw[treeedge] (p4) -- (v4) -- (q3) -- (q2) -- (q1) -- (l4);
				\draw[treeedge] (q3) -- (l1);
				\draw[treeedge] (q2) -- (l2);
				\draw[treeedge] (q1) -- (l3);
				
				\draw[treeedge] (q4) -- (v4);
				
			\end{tikzpicture}
		}
		\end{tabular}
		\caption{$\cM_5$}
		\label{fig:Mnetwork}
	\end{subfigure}
	\hfill 
	\begin{subfigure}[b]{0.30\textwidth}
		\centering
		\begin{tabular}[b]{c}
			\resizebox{\textwidth}{\textwidth}{%
				
			\begin{tikzpicture}[
				scale=0.8,
				treevertex/.style={circle, fill=black, inner sep=1.8pt},
				leafvertex/.style={circle, fill=black, inner sep=1.8pt},
				treeedge/.style={thick}
				]
				\node[treevertex] (r) at (2.5,4) {};
				\node[treevertex] (i1) at (2.0,3.2) {};
				\node[treevertex] (i2) at (1.5,2.4) {};
				\node[treevertex] (i3) at (1.0,1.6) {};
				
				\node[leafvertex, label=below:$\ell_5$] (l5) at (4.0,0) {};
				\node[leafvertex, label=below:$\ell_3$] (l1) at (3.0,0) {};
				\node[leafvertex, label=below:$\ell_2$] (l2) at (2.0,0) {};
				\node[leafvertex, label=below:$\ell_1$] (l3) at (1.0,0) {};
				\node[leafvertex, label=below:$\ell_4$] (l4) at (0.0,0) {};
				
				\draw[treeedge] (r) -- (i1) -- (i2) -- (i3);
				\draw[treeedge] (r) -- (l5);
				\draw[treeedge] (i1) -- (l1);
				\draw[treeedge] (i2) -- (l2);
				\draw[treeedge] (i3) -- (l3);
				\draw[treeedge] (i3) -- (l4);
			\end{tikzpicture}
		}
		\end{tabular}
		\caption{$\cC_5$}
	\end{subfigure}
	\caption{(I) Tree-child network~$\cN_5$, (II) a tree-child network~$\cM_5$ obtained by iteratively removing~$(p_i,v_i)$ for~$i\in\{1,2,3\}$, (III) a phylogenetic tree~$\cC_5$ which is displayed twice by~$\cN_5$.}
	\label{spinal}
\end{figure}
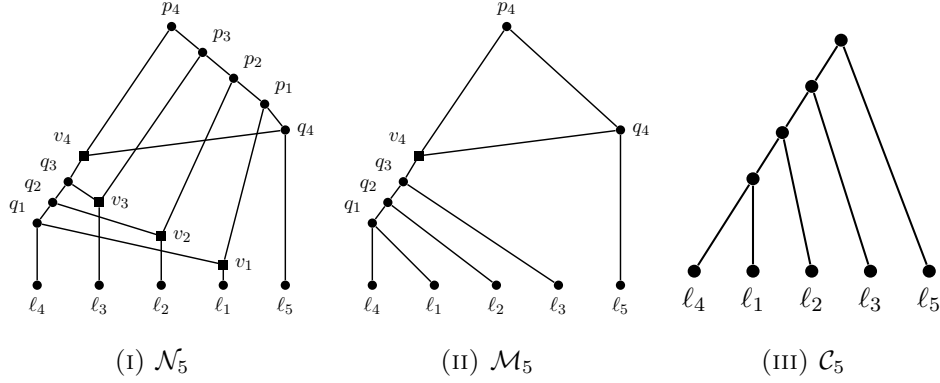

The proof of Theorem~\ref{main} requires us to show that, for each $n\ge 2$, there is a tree-child network with $n$ leaves and $n-1$ reticulations which displays $2^{n-1}-1$ rooted phylogenetic trees. To this end, we begin this section by constructing, for all $n\ge 2$, a tree-child network $\cN_n$ on $n$ leaves whose number of displayed rooted phylogenetic trees totals this upper bound. Illustrations of $\cN_5$ and the unique rooted phylogenetic tree~$\cC_5$ which is displayed twice by~$\cN_5$ are shown in Figure~\ref{spinal}.

For each $n\ge 2$, let $P_n$ denote the directed path
\begin{equation*}
p_{n-1},p_{n-2},\ldots, p_1, q_{n-1}, v_{n-1}, q_{n-2},q_{n-3}, \ldots, q_1, \ell_{n-1}.
\end{equation*}
The vertex $v_{n-1}$ will eventually be a reticulation, while $p_{n-1}$ and $\ell_{n-1}$ will be the root and a leaf, respectively, of $\cN_n$. 
For each $i\in \{1, 2, \ldots, n-2\}$, add a new vertex $v_i$ and adjoin $v_i$ to $P_n$ with the arcs $(p_i, v_i)$ and $(q_i, v_i)$, and adjoin a new leaf $\ell_i$ to $v_i$ with the arc $(v_i, \ell_i)$. Finally, add the arcs~$(p_{n-1},v_{n-1})$ and $(q_{n-1},\ell_n)$.
For all $n\ge 2$, let $\cN_n$ denote the resulting directed graph. It is easily checked that $\cN_n$ is tree-child with $n$ leaves and $n-1$ reticulations. Furthermore, for each $n\ge 2$, let $\cC_n$ denote the rooted phylogenetic tree realised by the two distinct embeddings containing the reticulation arcs in
$$\{(p_1, v_1), (p_2, v_2), \ldots, (p_{n-2}, v_{n-2})\}.$$
It is easily seen that $\cC_n$ is well defined. We now prove Theorem~\ref{main}.

\begin{proof}[Proof of Theorem~\ref{main}]
Let $n\ge 2$. If $\cN$ has at most~$n-2$ reticulations, then $|T(\cN)|\le 2^{n-2}\le 2^{n-1}-1$ and the inequality holds. Thus we may assume that $\cN$ has exactly~$n-1$ reticulations. We prove the inequality in the theorem by induction on the number $n$ of leaves of $\cN$. If $n=2$, then there is exactly one rooted phylogenetic tree on the same set of two leaves, and so $|T(\cN)|\le 1$ and the inequality holds. If $n=3$, then, as there are exactly three distinct phylogenetic trees on the same set of three leaves, $|T(\cN)|\le 3$ and, again, the inequality holds. Now suppose that $n\ge 4$ and the inequality holds for all tree-child networks with at least two leaves and fewer than $n$ leaves.

By Lemma~\ref{cherry}, $\cN$ has either a cherry or a reticulated cherry. If~$\cN$ has a cherry, say~$\{a, b\}$, then $\cN\ba b$ is a tree-child network with $n-1$ leaves and $n-1$ reticulations, contradicting Proposition~\ref{n-1}.
Thus $\cN$ has a reticulated cherry, say $(b, a)$, where $b$ is the reticulation leaf. Let $p_a$ and $p_b$ denote the parents of $a$ and $b$, respectively, and let $g_b$ denote the parent of $p_b$ that is not $p_a$. By Lemma~\ref{deletion}, $\cN\ba (g_b, p_b)$ is tree-child and so, as $\{a, b\}$ is a cherry of $\cN\ba (g_b, p_b)$, it follows that
$$\cN' = \big(\cN\ba (g_b, p_b)\big)\ba b$$
is also tree-child. Since $\cN'$ is a tree-child network with~$n-1$ leaves and $n-2$ reticulations, it follows by the induction assumption that $|T(\cN')|\le 2^{n-2}-1$.
This implies that~$\cN'$ displays a rooted phylogenetic tree~$\cT'$ twice. Let $E'_1$ and $E'_2$ be the two distinct embeddings realising $\cT'$. Consider the rooted phylogenetic tree~$\cT$ obtained from~$\cT'$ by subdividing the arc incident to the leaf~$a$ with a vertex~$p_a$, and adjoining a new leaf $b$ with the arc~$(p_a, b)$.

We claim that~$\cN$ displays~$\cT$ twice. Let $p'_a$ denote the parent of $a$ in $\cN'$. For each $i\in \{1, 2\}$, let
$$E_i=\big(E'_i-\{(p'_a, a)\}\big)\cup \{(p'_a, p_a), (p_a, a), (p_a, p_b), (p_b, b)\}.$$
Since $(p'_a, a)\in E'_1\cap E'_2$, the arc $(p_a,p_b)$ is an arc in $\cN$, and $E'_1$ and $E'_2$ are two distinct embeddings of $\cT'$ in $\cN'$, it follows that $E_1$ and $E_2$ are two distinct embeddings of $\cT$ in $\cN$.
Thus~$|T(\cN)| \le 2^{n-1} - 1$, thereby establishing the 
inequality in Theorem~\ref{main}.

We next show that, for all $n\ge 2$, we have $|T(\cN_n)| = 2^{n-1}-1$.
Let $n\ge 2$ be a positive integer. Let~$\cS$ denote the set of all embeddings of all the phylogenetic trees displayed by $\cN_n$.


\begin{sublemma}\label{cla:DiffEmbeddings}
Let~$E_1$ and $E_2$ be distinct embeddings in $\cS$, and suppose that, for some $i\in\{1, 2, \ldots, n-2\}$, we have $(p_i, v_i)\in E_1$ and $(q_i, v_i)\in E_2$. Then the rooted phylogenetic trees realised by $E_1$ and $E_2$ are non-isomorphic.
\end{sublemma}

To see (\ref{cla:DiffEmbeddings}), let $\cT_1$ and $\cT_2$ be the rooted phylogenetic trees realised by $E_1$ and $E_2$. Then $(\ell_i, \ell_{n-1}, \ell_n)$ is a rooted triple of $\cT_2$, and either $(\ell_i, \ell_n, \ell_{n-1})$ or $(\ell_{n-1}, \ell_n, \ell_i)$ is a rooted triple of $\cT_1$ depending on whether $(p_{n-1}, v_{n-1})\in E_1$ or $(q_{n-1}, v_{n-1})\in E_1$. Thus $\cT_1\not\cong \cT_2$, completing the proof of (\ref{cla:DiffEmbeddings}).


Now let~$E_1,E_2\in \cS$ be distinct embeddings, and let $\cT_1$ and $\cT_2$ be the rooted phylogenetic trees realised by $E_1$ and $E_2$, respectively. It follows by (\ref{cla:DiffEmbeddings}) that $\cT_1\not\cong \cT_2$ unless
$$(E_1-E_2)\cup (E_2-E_1)=\{(p_{n-1}, v_{n-1}), (q_{n-1}, v_{n-1})\}.$$
Without loss of generality, we may assume that $(p_{n-1}, v_{n-1})\in E_1$ and $(q_{n-1}, v_{n-1})\in E_2$. If, for all $i\in \{1, 2, \ldots, n-2\}$, we have $(q_i, v_i)\in E_1\cap E_2$, then $\cT_1\cong \cT_2\cong \cC_n$ and $(\ell_i, \ell_{n-1}, \ell_n)$ is a rooted triple of $\cT_1$. Otherwise, for some $i\in \{1, 2, \ldots, n-2\}$, we have $(p_i, v_i)\in E_1\cap E_2$. But then $(\ell_i, \ell_n, \ell_{n-1})$ is a rooted triple of $\cT_1$ and $(\ell_{n-1}, \ell_n, \ell_i)$ is a rooted triple of $\cT_2$, and so $\cT_1\not\cong \cT_2$, and neither $\cT_1$ nor $\cT_2$ is isomorphic to $\cC_n$ in this case.
Hence $|T(\cN_n)|=2^{n-1}-1$. This completes the proof of the theorem.
\end{proof}

We next establish a lemma to help us find the rooted phylogenetic tree displayed twice by a tree-child network $\cN$ with~$n$ leaves and~$n-1$ reticulations in which $|T(\cN)|=2^{n-1}-1$.

\begin{lemma}\label{thm:FindDuplicateTree}
	Let~$\cN$ be a tree-child network with~$n$ leaves and~$n-1$ reticulations, where~$n\ge 3$.
	Suppose that~$|T(\cN)| = 2^{n-1}-1$.
	Let~$(b,a)$ be a reticulated cherry of~$\cN$, and let~$g_b$ be the parent of~$p_b$ which is not~$p_a$, where $p_a$ and $p_b$ are the parents of $a$ and $b$, respectively. 
	Then the number of distinct rooted phylogenetic trees displayed by $(\cN\ba (g_b, p_b))\ba b$ is  $2^{n-2}-1$.
\end{lemma}

\begin{proof}
	Let~$\cN'=(\cN\ba (g_b, p_b))\ba b$. Let $\cT$ be the rooted phylogenetic tree displayed twice in $\cN$, and let $E_1$ and $E_2$ be the two embeddings of $\cT$ in $\cN$. We first show that $(p_a, p_b)\in E_1\cap E_2$. If $(g_b, p_b)\in E_1\cap E_2$, then, as $|T(\cN)|=2^{n-1}-1$ and $\cN$ displays exactly one rooted phylogenetic tree twice,  $\cN\ba (g_b, p_b)$ displays $2^{n-2}$ distinct rooted phylogenetic trees. As $\{a, b\}$ is a cherry of $\cN\ba (g_b, p_b)$, it follows that $\cN'$ displays $2^{n-2}$ distinct rooted phylogenetic trees, contradicting Theorem~\ref{main}. Also, if $(p_a, p_b)\in E_1$ and $(g_b, p_b)\in E_2$, then $\cN\ba (g_b, p_b)$, and therefore $\cN'$, displays $2^{n-2}$ distinct rooted phylogenetic trees, which gives another contradiction. Thus $(p_a, p_b)\in E_1\cap E_2$, and so $\cN'$ displays $2^{n-2}-1$ distinct rooted phylogenetic trees.
	This completes the proof of the lemma.
\end{proof}

Let~$\cN$ be a tree-child network on~$n$ leaves and~$n-1$ reticulations, where~$n\ge3$.
To find the rooted phylogenetic tree displayed twice by~$\cN$, we repeatedly find a reticulated cherry, $(b, a)$ say, and delete the reticulation arc directed into the parent $p_b$ of $b$ that is not $(p_a, p_b)$, where $p_a$ is the parent of $a$, and then delete one of the leaves of the resulting cherry. Repeating this process, we eventually obtain a tree-child network with two leaves and one reticulation (which is necessary part of an underlying $3$-cycle). 
We now return to $\cN$ and delete every reticulation arc that was deleted in this process, thereby obtaining a tree-child network $\cM$ with $n$ leaves and one reticulation. In the latter process, we obtain a sequence of tree-child networks on~$n$ leaves and~$k$ reticulations for~$k=n-1,n-2,\ldots, 1$, where each such network displays exactly~$2^{k}-1$ trees (by invoking Lemma \ref{thm:FindDuplicateTree} each time).
The tree-child network $\cM$ displays exactly one rooted phylogenetic tree, and so its one reticulation must necessarily be part of an underlying $3$-cycle. In summary, we have the following proposition.

\begin{proposition}\label{cor:DuplicateTree}
	Up to isomorphism, the unique rooted phylogenetic tree displayed by $\cM$ is the rooted phylogenetic tree displayed twice by $\cN$.
\end{proposition}

To illustrate, consider the network~$\cN_5$ in Figure~\ref{spinal}. We remove the edges~$(p_i,v_i)$ for all $i\in\{1,2,3\}$ to obtain the network~$\cM_5$.
Note that~$\cM_5$ displays~$\cC_5$, and thus by Proposition~\ref{cor:DuplicateTree},~$\cN_5$ displays~$\cC_5$ twice.
One can also see this directly, by considering the two embeddings of~$\cC_5$ into~$\cN_5$, which contains the reticulation arcs in the sets $\{(q_1,v_1),(q_2,v_2),(q_3,v_3),(p_4,v_4)\}$ and~$\{(q_1,v_1),(q_2,v_2),(q_3,v_3),(q_4,v_4)\}$, respectively.


\begin{thebibliography}{88}
\bibitem{bor16} M.\ Bordewich, C.\ Semple, Determining phylogenetic networks from inter-taxa distances, Journal of Mathematical Biology 73 (2016) 283--303.

\bibitem{car09} G.\ Cardona, F.\ Rossell\'{o}, G.\ Valiente, Comparison of tree-child phylogenetic networks, IEEE/ACM Transactions on Computational Biology and Bioinformatics 6 (2009) 552--569.

\bibitem{doe21} J.\ D\"{o}cker, S.\ Linz, C.\ Semple, Display sets of normal and tree-child networks, Electronic Journal of Combinatorics 28 (2021) \#P1.8.

\bibitem{doo99} W.F.\ Doolittle, Phylogenetic classification and the universal tree, Science 284 (1999) 2124--2129.

\bibitem{fra25} A.\ Francis, ``Normal'' phylogenetic networks may be emerging as the leading class, Journal of Theoretical Biology 614 (2025) 112236.

\bibitem{lin13} S.\ Linz, K. St John, C.\ Semple, Counting trees in a phylogenetic network is \#P-complete, SIAM Journal on Computing 42 (2013) 1768--1776.

\bibitem{sem16} C.\ Semple, Phylogenetic networks with every embedded phylogenetic tree a base tree, Bulletin of Mathematical Biology 78 (2016) 132--137.

\bibitem{sem26} C.\ Semple, K.\ Wicke, A sharp lower bound for the number of phylogenetic trees displayed by a tree-child network, Annals of Combinatorics, in press.

\bibitem{wil08} S.\ Willson, Reconstruction of certain phylogenetic networks from the genomes at their leaves, Journal of Theoretical Biology 252 (2008) 338--349.

\bibitem{wil10} S.\ Willson, Properties of normal phylogenetic networks, Bulletin of Mathematical Biology 72 (2010) 340--358.
\end{thebibliography}
\end{document}